\documentclass{amsart}
\usepackage{graphicx} 
\usepackage{amsmath}
\usepackage{amssymb}
\usepackage{geometry}
\usepackage{geometry}
\usepackage{mathtools}
\usepackage{braket}
\usepackage{xcolor}
\usepackage{amsthm}
\author{Yiming Chen}

\newtheorem{theorem}{Theorem}
\newtheorem{lemma}{Lemma}

\newtheorem{remark}{Remark}
\newtheorem{question}{Question}

\begin{document}
\title{Low-Degree Fourier Threshold for Random Boolean Functions}
\address{School of Mathematical Sciences, Peking University}
\email{ymchenmath@math.pku.edu.cn}
\date{}
\maketitle

\begin{abstract}


We study whether a uniformly random Boolean function
$f : \{-1,1\}^p \to \{-1,1\}$ is determined by its Walsh--Fourier coefficients
of degree at most $d$. We show that the threshold lies at $p/2$ up to an
$O(\sqrt{p \log p})$ window: if
\[
d \le \frac{p}{2} - \sqrt{\frac{p}{2}\bigl(\log p + \omega(1)\bigr)},
\]
then with probability $1-o(1)$ there exists another Boolean function
$g \ne f$ with the same degree-$\le d$ coefficients. Conversely, for every
fixed $\eta \in (0,1)$, if
\[
d \ge \frac{p}{2} + \sqrt{\frac{p}{2}\log\frac{6p}{\eta^2}},
\]
then with probability at least $1-2^{-p}$, the function $f$ is uniquely
determined by its degree-$\le d$ coefficients, even among all bounded
functions $g : \{-1,1\}^p \to [-1,1]$. This resolves a question of
Vershynin~\cite{Vershynin2024LowFrequencies}.

\end{abstract}

\section*{Introduction}

An interesting problem in high-dimensional discrete analysis is whether local information determines a global object. On the Boolean cube $\{-1,1\}^p$, this can be stated equivalently in two ways: probabilistically, as whether a distribution is determined by its low-dimensional marginals; and Fourier-analytically, as whether a function is determined by its Walsh--Fourier coefficients of degree at most $d$ \cite{ODonnell2014}. These perspectives are equivalent under the indicator-function correspondence used by Vershynin \cite{Vershynin2024LowFrequencies}: fixing all marginals up to dimension $d$ is equivalent to fixing all Walsh coefficients indexed by sets $J\subset[p]$ with $|J|\le d$. Related themes also appear in pseudorandomness, where \(k\)-wise independent distributions match low-order marginals and \(\varepsilon\)-biased spaces control parity statistics~\cite{BenjaminiGurelGurevichPeled2012,NaorNaor1993}.


Motivated by this observation, Vershynin~\cite{Vershynin2024LowFrequencies} made further progress toward the uniqueness question. In detail, he proved that

\begin{theorem}[\cite{Vershynin2024LowFrequencies}]\label{T-1}
For almost every Boolean function
\[
f:\{-1,1\}^p \to \{-1,1\},
\]
there exists a bounded function
\[
f':\{-1,1\}^p \to (-1,1)
\]
such that $f'$ has the same Fourier coefficients as $f$ in all degrees up to
\[
(1/2-o(1))p.
\]
\end{theorem}


This result places the problem in an instructive contrast with several classical positive identifiability results. Already in 1961, Chow proved that a Boolean threshold function is uniquely determined by its degree-$0$ and degree-$1$ Fourier coefficients \cite{Chow1961}; see also the later algorithmic treatment of O'Donnell and Servedio \cite{OdonnellServedio2011}. More generally, degree-$d$ polynomial threshold functions are known to be robustly identifiable from their degree-$d$ Chow parameters, i.e. from their Fourier coefficients of degree at most $d$, even within the class of bounded functions \cite{DiakonikolasKane2018}. Theorem \ref{T-1} shows that this behavior is highly nongeneric. In the absence of special algebraic or geometric structure, low-degree Fourier data typically do \emph{not} determine a Boolean function until one keeps almost half of all coordinate dimensions.

Therefore, uniqueness fails if one allows non-Boolean competitors. Naturally, the remaining 
issue is whether such a competitor can also be chosen Boolean, which leads to the following question.

\medskip
\begin{question}\label{q1}
\textit{Is it true that most Boolean functions are not determined by their Fourier 
coefficients up to almost half the dimension? Is half the dimension the optimal 
threshold?}
\end{question}


The purpose of this paper is to answer Question \ref{q1}.

\begin{theorem}\label{thm:main}
Let $f:\{-1,1\}^p \to \{-1,1\}$ be a uniformly random Boolean function.

\begin{enumerate}
    \item If
    \[
    d \le \frac{p}{2}-\sqrt{\frac{p}{2}\bigl(\log p+\omega(1)\bigr)},
    \]
    then with probability $1-o(1)$ there exists another Boolean function
    \[
    g:\{-1,1\}^p \to \{-1,1\}, \qquad g\neq f,
    \]
    such that $g$ and $f$ have the same Fourier coefficients in every degree at most $d$.

    \item Conversely, if
    \[
    d \ge \frac{p}{2}+\sqrt{\frac{p}{2}\log\frac{6p}{\eta^2}}
    \qquad \text{for fixed } \eta\in(0,1),
    \]
    then with probability at least $1-2^{-p}$, the function $f$ is uniquely determined by its Fourier coefficients in every degree at most $d$, even among all bounded functions
    \[
    g:\{-1,1\}^p \to [-1,1].
    \]
\end{enumerate}

\end{theorem}

\begin{remark}
Theorem 2 places the threshold at \(p/2\) up to a window of order \(\sqrt{p\log p}\). Moreover, the lower bound of Theorem \ref{thm:main} improved Theorem \ref{T-1} from $f':\{-1,1\}^p \to (-1,1)$ to $g:\{-1,1\}^p \to \{-1,1\}$. Thus, both questions in Question~\ref{q1} are answered.
\end{remark}

\section*{Setup}

Let \(\Omega := \{-1,1\}^p\) and \(N:=|\Omega|=2^p\). For \(J\subset[p]:=\{1,\dots,p\}\), define the Walsh function
\[
w_J(x):=\prod_{j\in J} x_j,
\qquad w_{\varnothing}\equiv 1.
\]
For any function \(f:\Omega\to\mathbb R\), define its Walsh--Fourier coefficient by
\[
\widehat f(J):=2^{-p}\sum_{x\in\Omega} f(x)w_J(x).
\]
For \(d\in\{0,1,\dots,p\}\), define the low-frequency data map
\[
\Phi_d(f):=\bigl(\widehat f(J)\bigr)_{J\subset[p],\,|J|\le d},
\]
and the combinatorial quantities
\[
K_d=\sum_{k=0}^d \binom{p}{k},
\qquad
M_d:=\sum_{k=d+1}^p \binom{p}{k}=N-K_d.
\]
For a Boolean function \(f:\Omega\to\{-1,1\}\), set
\[
B_d(f):=\{g:\Omega\to[-1,1]:\Phi_d(g)=\Phi_d(f)\}.
\]
Throughout, \(\log\) denotes the natural logarithm. When we say that \(f\) is a uniformly random Boolean function, we mean that the values \(f(x)\), \(x\in\Omega\), are independent Rademacher random variables.

\begin{lemma}[Binomial tail estimate; cf.~\cite{Hoeffding1963}]\label{lem:binomial-tail}
Let \(X\sim \mathrm{Bin}(p,1/2)\). Then, for every \(t>0\),
\[
\mathbb P\!\left(X\le \frac p2-t\right)\le e^{-2t^2/p}
\qquad\text{and}\qquad
\mathbb P\!\left(X\ge \frac p2+t\right)\le e^{-2t^2/p}.
\]
\end{lemma}

%

\begin{lemma}\label{lem:rademacher-tail}
Let \(\varepsilon_1,\dots,\varepsilon_m\) be independent Rademacher random variables and let \(a=(a_1,\dots,a_m)\in\mathbb R^m\). Then, for every \(\eta>0\),
\[
\mathbb P\!\left(\left|\sum_{i=1}^m a_i\varepsilon_i\right|>\eta\right)
\le
2\exp\!\left(-\frac{\eta^2}{2\|a\|_2^2}\right).
\]
\end{lemma}
\subsection*{Proof strategy}  

The proof of the lower bound is essentially counting-theoretic. Each Fourier coefficient \(\widehat f(J)\) of a Boolean function takes at most \(2^p+1\) values, because it is the normalized sum of \(2^p\) signs. Therefore the complete low-frequency vector \((\widehat f(J))_{|J|\le d}\) can take at most \((2^p+1)^{K_d}\) distinct values. When \(K_d \ll 2^p/p\), this number is exponentially negligible compared with the total number \(2^{2^p}\) of Boolean functions. Hence collisions are unavoidable for the overwhelming majority of Boolean functions, and these collisions already occur inside the Boolean class.

The upper bound is proved by a direct truncation argument. Let \(P_df\) denote the degree-\(d\) Fourier truncation of \(f\). For each fixed point \(x\in\{-1,1\}^p\), the error \(f(x)-P_df(x)\) can be written as a Rademacher linear form in the independent values \(f(y)\), \(y\in\Omega\), and its variance is exactly
\[
2^{-p}\sum_{k=d+1}^p \binom pk.
\]
A pointwise subgaussian estimate followed by a union bound over all \(2^p\) points shows that, on the upper side of the threshold, one has \(\|P_df-f\|_\infty<1\) with overwhelming probability. Consequently, \(P_df\) has the same pointwise sign as \(f\). If \(g:\{-1,1\}^p\to[-1,1]\) has the same degree-\(\le d\) Fourier coefficients as \(f\), then \(g-f\) is orthogonal to every degree-\(\le d\) polynomial, and in particular to \(P_df\). On the other hand, the sign agreement between \(P_df\) and \(f\) forces every summand in
\[
\sum_x (g(x)-f(x))P_df(x)
\]
to be nonpositive. Since the sum is also zero by orthogonality, each summand must vanish, and hence \(g=f\) pointwise.

\section*{Proof of Lower Bound}


\begin{theorem}\label{prop:boolean-collision}
Choose \(f:\Omega\to\{-1,1\}\) uniformly at random among all Boolean functions on \(\Omega\). Then, for every \(d\in\{0,1,\dots,p\}\),
\[
\mathbb P\Big(\exists g:\Omega\to\{-1,1\},\ g\ne f,\ \Phi_d(g)=\Phi_d(f)\Big)
\ge
1-\frac{(N+1)^{K_d}}{2^N}.
\]
Consequently, if
\[
d\le \frac p2-\sqrt{\frac p2\bigl(\log p+\omega(1)\bigr)},
\]
then
\[
\mathbb P\Big(\exists g:\Omega\to\{-1,1\},\ g\ne f,\ \Phi_d(g)=\Phi_d(f)\Big)\to 1
\qquad (p\to\infty).
\]
In particular, the same conclusion holds for \(d=\lfloor cp\rfloor\) for every fixed constant \(c<1/2\).
\end{theorem}

\begin{proof}
For a fixed set \(J\subset[p]\), define
\[
S_J(f):=\sum_{x\in\Omega} f(x)w_J(x).
\]
Since each factor \(f(x)w_J(x)\) belongs to \(\{-1,1\}\), the quantity \(S_J(f)\) is a sum of \(N\) many \(\pm1\) terms. Therefore
\[
S_J(f)\in\{-N,-N+2,-N+4,\dots,N-2,N\},
\]
so \(S_J(f)\) has at most \(N+1\) possible values. 

Then we have
\[
\widehat f(J)=\frac{1}{N}S_J(f),
\]
it follows that each Fourier coefficient \(\widehat f(J)\) also has at most \(N+1\) possible values.

The vector \(\Phi_d(f)\) has exactly \(K_d\) coordinates, one for each \(J\subset[p]\) with \(|J|\le d\). Hence
\[
|\operatorname{Im}(\Phi_d)|\le (N+1)^{K_d}.
\]
Now let
\[
U_d:=\Bigl\{f\in\{-1,1\}^{\Omega}: \{g\in\{-1,1\}^{\Omega}:\Phi_d(g)=\Phi_d(f)\}=\{f\}\Bigr\}.
\]
Thus \(U_d\) is exactly the set of Boolean functions that are uniquely determined, among Boolean functions, by their low-frequency data up to degree \(d\). If \(f_1,f_2\in U_d\) and \(\Phi_d(f_1)=\Phi_d(f_2)\), then \(f_2\) belongs to the set
\[
\{g\in\{-1,1\}^{\Omega}:\Phi_d(g)=\Phi_d(f_1)\},
\]
which equals \(\{f_1\}\) by the definition of \(U_d\). Hence \(f_2=f_1\). Therefore the restriction of \(\Phi_d\) to \(U_d\) is injective, and so
\[
|U_d|\le |\operatorname{Im}(\Phi_d)|\le (N+1)^{K_d}.
\]
Since the total number of Boolean functions on \(\Omega\) is \(2^N\), we obtain
\[
\mathbb P(f\in U_d)=\frac{|U_d|}{2^N}\le \frac{(N+1)^{K_d}}{2^N}.
\]
Taking complements proves the first claim.

For the asymptotic claim, let \(X\sim\mathrm{Bin}(p,1/2)\). Then
\[
K_d=\sum_{k=0}^d \binom{p}{k}=N\,\mathbb P(X\le d).
\]
Assume that
\[
d\le \frac p2-t,
\qquad
t:=\sqrt{\frac p2\bigl(\log p+\omega(1)\bigr)}.
\]
By Lemma~\ref{lem:binomial-tail},
\[
\mathbb P(X\le d)
\le
\mathbb P\!\left(X\le \frac p2-t\right)
\le
\exp\!\left(-\frac{2t^2}{p}\right)
=
\exp\bigl(-\log p-\omega(1)\bigr).
\]
Therefore,
\[
K_d
\le
N\exp\bigl(-\log p-\omega(1)\bigr)
=
\frac{N}{p}e^{-\omega(1)}
=
 o\!\left(\frac{N}{p}\right).
\]
Consequently,
\[
\frac{(N+1)^{K_d}}{2^N}
=
\exp\bigl(K_d\log(N+1)-N\log 2\bigr).
\]
Since \(N=2^p\), we have \(N+1\le 2^{p+1}\), and thus
\[
K_d\log(N+1)
\le
(p+1)(\log 2)K_d
=
 o(N).
\]
Hence
\[
K_d\log(N+1)-N\log 2=-(1+o(1))N\log 2\to -\infty,
\]
which implies
\[
\frac{(N+1)^{K_d}}{2^N}\to 0.
\]
Substituting this into the first bound yields the desired convergence to \(1\).

Finally, if \(d=\lfloor cp\rfloor\) with a fixed \(c<1/2\), then
\[
\frac p2-d=(1/2-c)p+O(1)\gg \sqrt{p\log p},
\]
so the displayed hypothesis is satisfied for all sufficiently large \(p\).
\end{proof}

\section*{Proof of Upper Bound}
\begin{theorem}\label{thm:upper-unique}
Fix \(\eta\in(0,1)\). If
\[
d\ge \frac p2+\sqrt{\frac p2\log\frac{6p}{\eta^2}},
\]
then
\[
\mathbb P\bigl(B_d(f)=\{f\}\bigr)\ge 1-2^{-p}.
\]
Consequently, for every fixed \(\eta\in(0,1)\), if
\[
d\ge \frac p2+\sqrt{\frac p2\bigl(\log p+\omega(1)\bigr)},
\]
then
\[
\mathbb P\bigl(B_d(f)=\{f\}\bigr)\to 1
\qquad (p\to\infty).
\]
\end{theorem}

\begin{proof}
For this proof, it is convenient to work on \(\mathbb R^{\Omega}\) with the Euclidean inner product
\[
\langle u,v\rangle:=\sum_{x\in\Omega} u(x)v(x).
\]
For each \(J\subset[p]\), define
\[
\varphi_J:=N^{-1/2}w_J.
\]
Since the Walsh functions are orthogonal in normalized \(L^2(\Omega)\), the family \((\varphi_J)_{J\subset[p]}\) is an orthonormal basis of \((\mathbb R^{\Omega},\langle\cdot,\cdot\rangle)\).

Let
\[
V_d:=\operatorname{span}\{\varphi_J:|J|\le d\},
\]
and recall that \(P_d\) denote the orthogonal projection onto \(V_d\). Define
\[
q_d:=P_df,
\qquad
r_d:=f-q_d=(I-P_d)f.
\]
Note that
\[
\langle f,\varphi_J\rangle=N^{1/2}\widehat f(J),
\]
we have
\[
q_d(x)=\sum_{J\subset[p],\,|J|\le d} \widehat f(J)w_J(x)
\qquad \text{for every } x\in\Omega.
\]
Thus \(q_d\) is exactly the degree-\(d\) Walsh--Fourier truncation of \(f\).

Fix \(x\in\Omega\), and let \(e_x\in\mathbb R^{\Omega}\) be the standard basis vector at \(x\). Since \(I-P_d\) is self-adjoint,
\[
r_d(x)
=
\langle e_x,r_d\rangle
=
\langle e_x,(I-P_d)f\rangle
=
\langle (I-P_d)e_x,f\rangle.
\]
Set
\[
a^{(x)}:=(I-P_d)e_x\in\mathbb R^{\Omega}.
\]
Then
\[
r_d(x)=\sum_{y\in\Omega} a^{(x)}_y f(y).
\]
Since the random variables \(f(y)\), \(y\in\Omega\), are independent Rademacher variables, Lemma~\ref{lem:rademacher-tail} yields
\[
\mathbb P\bigl(|r_d(x)|>\eta\bigr)
\le
2\exp\!\left(-\frac{\eta^2}{2\|a^{(x)}\|_2^2}\right).
\]
We now compute \(\|a^{(x)}\|_2^2\). Recall that \((\varphi_J)_{J\subset[p]}\) is an orthonormal basis and
\[
\langle e_x,\varphi_J\rangle=N^{-1/2}w_J(x),
\]
Parseval's identity gives
\[
\|(I-P_d)e_x\|_2^2
=
\sum_{J\subset[p],\,|J|>d} |\langle e_x,\varphi_J\rangle|^2
=
\sum_{J\subset[p],\,|J|>d} \frac{1}{N}
=
\frac{M_d}{N}.
\]
Therefore,
\[
\mathbb P\bigl(|r_d(x)|>\eta\bigr)
\le
2\exp\!\left(-\frac{\eta^2N}{2M_d}\right).
\]
Taking the union bound over all \(x\in\Omega\), we obtain
\[
\mathbb P\bigl(\|r_d\|_{\infty}>\eta\bigr)
\le
2N\exp\!\left(-\frac{\eta^2N}{2M_d}\right).
\]

It remains to control \(M_d\). Let \(X\sim\mathrm{Bin}(p,1/2)\). Then
\[
M_d=\sum_{k=d+1}^p \binom{p}{k}=N\,\mathbb P(X>d).
\]
By Lemma~\ref{lem:binomial-tail} and the assumption on \(d\),
\[
\frac{M_d}{N}
=
\mathbb P(X>d)
\le
\mathbb P\!\left(X-\frac p2\ge \sqrt{\frac p2\log\frac{6p}{\eta^2}}\right)
\le
\exp\!\left(-\log\frac{6p}{\eta^2}\right)
=
\frac{\eta^2}{6p}.
\]
Hence
\[
\mathbb P\bigl(\|r_d\|_{\infty}>\eta\bigr)
\le
2N e^{-3p}
=
2^{p+1}e^{-3p}
\le
2^{-p}.
\]
(The last inequality holds for every \(p\ge 1\).)
Thus, with probability at least \(1-2^{-p}\), we have
\[
\|q_d-f\|_{\infty}=\|r_d\|_{\infty}\le \eta.
\]

Now fix a realization of \(f\) for which this bound holds. Since \(\eta<1\) and \(f(x)\in\{-1,1\}\), we have for every \(x\in\Omega\):
\[
f(x)=1 \implies q_d(x)\ge 1-\eta>0,
\qquad
f(x)=-1 \implies q_d(x)\le -1+\eta<0.
\]
Therefore,
\[
\operatorname{sign}(q_d(x))=f(x)
\qquad\text{and}\qquad
q_d(x)\ne 0
\qquad\forall x\in\Omega.
\]

Let \(g\in B_d(f)\), and define \(h:=g-f\). Since \(\Phi_d(g)=\Phi_d(f)\), we have
\[
\widehat h(J)=0
\qquad\text{for all } J\subset[p] \text{ with } |J|\le d.
\]
Thus
\[
\langle h,\varphi_J\rangle=N^{1/2}\widehat h(J),
\]
this means that \(h\perp V_d\) in the Euclidean inner product. Since \(q_d\in V_d\), it follows that
\[
0=\langle h,q_d\rangle
=
\sum_{x\in\Omega} (g(x)-f(x))q_d(x).
\]
For each \(x\in\Omega\), the summand is nonpositive. Indeed:
\begin{itemize}
\item if \(f(x)=1\), then \(q_d(x)>0\) and \(g(x)-f(x)=g(x)-1\le 0\), so \((g(x)-f(x))q_d(x)\le 0\);
\item if \(f(x)=-1\), then \(q_d(x)<0\) and \(g(x)-f(x)=g(x)+1\ge 0\), so again \((g(x)-f(x))q_d(x)\le 0\).
\end{itemize}
Thus every term in the finite sum is nonpositive, and the sum is equal to \(0\). Therefore every term must vanish:
\[
(g(x)-f(x))q_d(x)=0
\qquad\forall x\in\Omega.
\]
Since \(q_d(x)\ne 0\) for every \(x\), we conclude that \(g(x)=f(x)\) for every \(x\in\Omega\). Hence \(g=f\), and therefore
\[
B_d(f)=\{f\}.
\]
This proves the first statement. The asymptotic statement follows immediately because, for fixed \(\eta\in(0,1)\),
\[
\log\frac{6p}{\eta^2}=\log p+O(1).
\]
\end{proof}

\end{document}